\documentclass[a4paper]{article}

\let\ucal=\mathcal

\usepackage[leqno]{amsmath}
\usepackage{amsthm}
\usepackage{cite}
\usepackage{amssymb}
\usepackage{mathrsfs}
\usepackage{eucal}
\newtheorem{pro}{Proposition} 
\newtheorem{lem}[pro]{Lemma}
\newtheorem{thm}[pro]{Theorem}
\newtheorem{corl}[pro]{Corollary}
\theoremstyle{remark}
\newtheorem{rem}[pro]{Remark}
\theoremstyle{definition}
\newtheorem{deff}[pro]{Definition}

\numberwithin{equation}{subsubsection}

\def\toRomans{
  \renewcommand{\theenumi}{(\roman{enumi})}}

\toRomans

\allowdisplaybreaks[1]

\title{Unbounded Fredholm modules and double operator integrals}
\author{D.~Potapov \and F.~Sukochev}

\def\aA{{\mathcal{A}}}

\def\aMtilde{{\tilde \aM}}

\def\Bd{B}
\def\cC{{\ucal{C}}}
\def\aM{{\mathcal{M}}}

\def\sL{{\mathcal{L}}}

\def\Z{{\mathbb{Z}}}
\def\sE{{\mathcal{E}}}
\def\sH{{\mathcal{H}}}
\def\Rl{{\mathbb{R}}}
\def\Cx{{\mathbb{C}}}

\def\dom#1{{\mathcal{D}{\rm om}~#1}}
\def\dtimes{{\times\kern-1pt\times}}

\def\sEcross{{\sE^\times}}

\let\tilde=\widetilde

\def\Dom#1{{\mathscr{D}(#1)}}

\begin{document}

\maketitle

\bibliographystyle{short}

\parskip=0pt plus 1pt minus 0pt

\begin{abstract}
  In noncommutative geometry one is interested in invariants such as the
  Fredholm index or spectral flow and their calculation using cyclic
  cocycles.  A variety of formulae have been established under side
  conditions called summability constraints. These can be formulated in
  two ways, either for spectral triples or for bounded Fredholm modules.
  We study the relationship between these by proving various properties
  of the map on unbounded self adjoint operators $D$ given by $f
  (D)=D(1+D^2)^{-1/2}$. In particular we prove commutator estimates
  which are needed for the bounded case.  In fact our methods work in
  the setting of semifinite noncommutative geometry where one has $D$ as
  an unbounded self adjoint linear operator affiliated with a
  semi-finite von Neumann algebra $\aM$.  More precisely we show that
  for a pair $D,D_0$ of such operators with $D-D_0$ a bounded
  self-adjoint linear operator from $\aM$ and $ ({\bf
    1}+D_0^2)^{-1/2}\in \sE$, where $\sE$ is a noncommutative symmetric
  space associated with $\aM$, then $$ \Vert f(D) - f (D_0) \Vert_{\sE}
  \leq C\cdot \Vert D-D_0\Vert_{\aM}. $$ This result is further used to
  show continuous differentiability of the mapping between an odd
  $\sE$-summable spectral triple and its bounded counterpart.
\end{abstract}

\subsubsection{Introduction}

\def\sgn{{\mathrm{sgn}\,}}

This paper concerns questions arising in noncommutative geometry in
general and the study of spectral flow in particular. The basic issues
were first exposed in A.~Connes~\cite{Connes1985-MR823176,
  Connes1994-MR1303779}.  The object of study is a spectral triple
which consists of a separable Hilbert space $\sH$, a densely defined
unbounded self-adjoint linear operator $D$ and a *-algebra of
bounded operators on $\sH$ such that $[D,a]$ extends to a bounded
operator on $\sH$ for all $a\in \aA$. If there is a grading operator
$\gamma$ (that is $\gamma$ is self adjoint and $\gamma^2=1$) which
anticommutes with $D$ then the spectral triple is said to be even and
otherwise it is odd. As the grading operator will not play a role in
what we do here we will ignore it in the sequel. We note however that
spectral flow, which will form a major application of our results, is
non-trivial only in the odd case.

In order to construct formulae for spectral flow or the Fredholm index
one employs explicit cyclic cocycles whose existence requires
`summability conditions' on $D$. These take the form of specifying
some symmetrically normed ideal $\sE$ of compact operators on $\sH$
and requiring $(1+D^2)^{-1/2}\in \sE$. In \cite{Connes1994-MR1303779}
three cases arise naturally namely the Schatten ideals $\sL^p$ (the
$p$-summable case), the ideal $Li$ which is relevant to so-called
theta summable spectral triples and the ideal $\sL^{p,\infty}$ which is
naturally associated to the Dixmier trace.

In constructing formulae for cyclic cocycles one is faced with deciding
when a given cocycle is in the cohomology class of the Chern character
\cite{Connes1994-MR1303779}. This Chern character is defined not for
spectral triples but for bounded Fredholm modules for $\aA$.
The passage from unbounded to bounded requires us to study the map
$D\to F_D=D(1+D^2)^{-1/2}$. The definition of an
$\sE$-summable bounded Fredholm module
requires the commutator $[F_D,a]$ to be in $\sE$ so that we want to know
when this follows from the assumption $(1+D^2)^{-1/2}\in \sE$.
This explains the need for methods to prove
commutator estimates which generalize earlier work.


In the setting of Schatten-von Neumann ideals, it was established
in~\cite{CarPhi1998-MR1638603} (respectively,
in~\cite{Sukochev2000-MR1767406}) that if~$(1+D_0^2)^{- \frac 12} \in
\sL^q$, $q < p$ (respectively,~$(1+D^2)^{-\frac 12} \in Li^\beta$,
$\beta > \alpha$), then we have the Lipschitz estimates in~$\sL^p$
(respectively~$Li^\alpha$).  The sharp commutator estimate in this
setting was proved in~\cite{ScWaWa1998}.

Beginning in~\cite{CarPhi1998-MR1638603} and continuing
in~\cite{CaPhSu2000-MR1758245,CaPhSu2003-MR1954456,
  BenFac2006-MR2186918, CarPhi2004-MR2053481, CaPhRe2004-MR2069783,
  CPRS2, CPRS3, CaPhReSu-III, BCPRSW} an extension is
made to the framework of noncommutative geometry described
in~\cite{Connes1994-MR1303779}. This extension is to the situation where
we take $\aM$ be a semi-finite von Neumann algebra acting on $\sH$ with
normal self-adjoint faithful trace~$\tau$, we let~$D$ be an unbounded
self-adjoint linear operator affiliated with~$\aM$ and take $\sE$ to be
a noncommutative symmetric space of~$\tau$-measurable operators (all
these notions are explained in the next Section).  Summability in this
setting means $(1+D^2)^{-1/2}\in \sE$ however now the passage from the
unbounded to bounded picture is a much more complex issue.  Our
systematic approach to these questions results in a general approach
which we illustrate in Theorems 11 and 18.  We establish in particular
that for the case of a general semifinite von Neumann algebra from the
condition $(1+D^2)^{-1/2}\in \sE$ there follows the bound $$ \| [ F_D ,
a] \|_\sE \leq c\, \|[D, a]\|. $$

A related question arises in~\cite{Phillips1997-MR1478707,
  CarPhi1998-MR1638603, CarPhi2004-MR2053481} where the notion of
spectral flow along a path of unbounded self adjoint operators
affiliated to $\aM$ is studied and analytic formulae to calculate
spectral flow are given. In~\cite{CPRS2} and~\cite{CPRS3} a local
index formula for spectral flow (analogous to the formula of
Connes-Moscovici~\cite{ConMos1995-MR1334867} for the case where $\aM$
is the bounded operators on $\sH$) is given and its relation to the
Chern character of a `Breuer-Fredholm module' studied.  In all of this
work the properties of the function
\begin{equation}
  \label{MainFuncIntro}
  f(t) = \frac t {(1+ t^2)^{\frac 12}},\ \ t \in
  \Rl
\end{equation}
defined on unbounded self adjoint Breuer-Fredholm operators plays an
essential role.  In particular the question of operator
differentiability of $f$ arises.  Until now results about this
function have been proved in an ad hoc fashion and are restricted to
particular choices of the ideal $\sE$. More general ideals do need to
be studied as they arise in a very natural way once one begins a
deeper study of noncommutative geometry in this setting, see for
example~\cite{CaReSeSu}.

The principal result of our paper in this direction is that if~$(\aM,
D)$ is an odd $\sE$-summable semifinite spectral triple, then~$(\aM, f
(D))$ is an odd bounded $\sE$-summable (pre-)Breuer-Fredholm module and
furthermore the mapping $$ (\aM, D_0) \mapsto (\aM, f(D)) $$ is
Lipschitz continuous and continuously differentiable on the affine space
of bounded self adjoint perturbations of $D$ (where the perturbation
comes from $\aM$). The need for such a result is noted in \cite{Wahl}.

Consider the following example.  Let~$\sH = L^2[0,1]$ and let~$\aM
= \Bd(\sH)$ (i.e., the algebra of all bounded linear operators
on~$\sH$). Consider the operator~$D_0 = i \frac d {dt}$ with~$\dom
D_0$ given by the class of all absolutely continuous
functions~$\xi$ on~$[0,1]$ such that~$\xi(0) = \xi(1)$.  It is
well-known that~$\sigma(D_0) = \Z$ and $$ \left( 1 +
D_0^2\right)^{-\frac 12} \in \cC^{1, \infty}, $$ where~$\cC^{1,
\infty}$ is the weak~$L^1$ ideal (see Section~\ref{sec:AppWeakLp}
below).  Taking a path of operators $$ s \mapsto D_s = i \frac
d{dt} + V(t, s), $$ where the path of potentials~$s \mapsto
V(t,s)$ is continuously differentiable in~$L^\infty[0,1]$,
Theorem~\ref{CiPATH} implies that the path of operators~$s \mapsto
f(D_s)$ is continuously differentiable in~$\cC^{1,
  \infty}$.

Partial results of this nature were earlier obtained
in~\cite{CarPhi1998-MR1638603, CarPhi2004-MR2053481, ScWaWa1998,
  Sukochev2000-MR1767406, CaPhSu2000-MR1758245} however the methods
employed in those papers are not adequate to determine the sense in
which this mapping $f$ is smooth on operators.  This suggested to the
authors that there was a need for a more powerful method that could
answer this question in full generality.  The technique we describe
here is partly based on an approach to the calculus of functions of
operators known as the theory of `double operator integrals'. It has
only recently been developed for the general semi-finite von Neumann
algebras in~\cite{PS-DiffP, PSW-DOI} and its applications to Lipschitz
and commutator estimates of operator functions begun
in~\cite{PS-RFlow,PotapovThesis,PoSu}.

The organization of the paper is as follows.  We briefly outline the
theory of double operator integrals in Section~\ref{sec:DOI} where we
also prove a number of technical results needed to analyze the
behavior of the operator function~$f$ in our context.
Section~\ref{BasicSec} contains the main results concerning Lipschitz
and commutator estimates for the function~$f$ (and some other operator
functions) which occur in noncommutative geometry and the theory of
spectral flow.  Section~\ref{sec:AppWeakLp} contains a specialization
of the results given in Section~\ref{BasicSec} to important
applications to the case of weak $L_p$-spaces and answers a question
raised by A.~Carey in the context of applications to spectral flow.
Finally, the last section~\ref{FredholmSec} explains how results
presented in Section~\ref{BasicSec} can be further refined to prove
the differentiability of the mapping~$(\aM, D_0) \mapsto (\aM,
f(D_0))$.  We also indicate there important implications for the theory
of spectral flow which are motivated by the papers ~\cite{BCPRSW,
  CarPhi1998-MR1638603, CarPhi2004-MR2053481, Wahl}. The
strategy of our proof is straightforward and applies equally well to
von Neumann algebras of type~$I$ and~$II$.

\subsubsection{Preliminaries}
\label{PrelimSec}

Let~$\aM$ be a semi-finite von Neumann algebra acting on a Hilbert
space~$\sH$ and equipped with normal semi-finite faithful
trace~$\tau$. The identity in $\mathcal {M}$ is denoted by
$1$ and~$\|\cdot\|$ stands for the operator norm.  An
operator~$D: \Dom{D} \mapsto \sH$, with domain $\Dom{D}\subseteq
\mathcal {H}$, is called {\it affiliated\/} with~$\aM$ if and only if,
for every unitary~$u \in \aM'$, $u^* D u = D$, i.e.\ (i)~$u(\Dom{D})
\subseteq \Dom{D}$ and (ii)~$u^* D u(\xi) = D(\xi)$, for every~$\xi
\in \Dom{D}$. We use the standard notation $D \eta \aM$ to indicate
that the operator~$D$ is affiliated with~$\aM$. A closed and densely
defined operator $D\eta \aM$ is called $\tau$-measurable if for every
$\epsilon >0$ there exists an orthogonal projection $p\in \aM$ such
the $p({\mathcal {H}})\subseteq \Dom{D}$ and $\tau(1-p)<\epsilon$. The
set of all $\tau$-measurable operators is denoted $\aMtilde$.  The
set~$\aMtilde$ is a $*$-algebra (with respect to strong multiplication
and addition) complete in the measure topology.  We refer the reader
to~\cite{SZ-Lec-on-vNA, FK1986, Nelson1974-MR0355628} for more
details.

We recall from~\cite{FK1986} the notion of {\it generalized singular
  value function}. Given a self-adjoint operator $T$ in $\aMtilde$, we
denote by $E^T(\cdot )$ the spectral measure of $T$. For every~$T \in
\aMtilde$, $E^{\vert T\vert}(B)\in \aM $ for all Borel sets
$B\subseteq{\mathbb{R}}$, and there exists $s>0$ such that $\tau
(E^{\vert T\vert}(s,\infty))<\infty$. For $t\geq 0$, we define
$$\mu_t(T)=\inf\{s\geq 0 : \tau (E^{\vert T\vert}(s,\infty))\leq t\}.$$
The function $\mu(T):[0,\infty)\to [0,\infty ]$ is called the {\it
generalized singular value function} (or decreasing rearrangement)
of $T$; note that $\mu_{(\cdot)}(T)\in L_\infty$ if and only if $T\in \aM$.

Throughout the text let~$E=E(0, \infty)$ be a symmetric Banach function
space, i.e.\ $E=E(0, \infty)$ is a rearrangement invariant Banach
function space on~$[0, \infty)$ (see~\cite{LT-II}).
We use the notation ~$g\prec \prec f$ to denote
submajorization in the sense of
Hardy, Littlewood and Polya, i.e.\ $$ \int_0^t \mu_s(g)\, ds \le
\int_0^t \mu_s(f) \, ds,\ \ t > 0.  $$
We will always require $E$ to have the additional
property that~$f, g\in E$ and~$g \prec\prec f$ implies that~$\|g\|_E \le
\|f\|_E$.

There is associated to each such space  $E$ a noncommutative symmetric
space~$\sE = E(\aM, \tau)$ defined by $$ \sE = \{ T \in\aMtilde,\ \
\mu_{(\cdot)}(T) \in E\}\ \ \text{with}\ \ \|T\|_\sE =
\|\mu_{(\cdot)}(T)\|_E. $$ If~$E = L^p$, $1\le p\le \infty$,
then~$\sL^p$ is the noncommutative~$L^p$-space.  For the sake of
brevity, we shall denote the norm in the space~$\sL^p$ by~$\|\cdot\|_p$.
Note, that the spaces~$\sL^\infty$ and~$\sL^1$ coincide with the
algebra~$\aM$ and the predual~$\aM_*$, respectively, and
that~$\|\cdot\|_\infty$ is the operator norm~$\|\cdot\|$.  We refer the
reader to~\cite{PSW-DOI, DoDode1989-MR1004176,DoDode1992-MR1188788} for
more information on noncommutative symmetric spaces.

The K\"othe dual~$\sEcross$ of a symmetric space~$\sE$ is the symmetric
space given by $$ \sEcross = \{ T\in\aMtilde:\ \ TS\in\sL^1,\ \
\text{whenever}\ \ S\in\sE $$ and
\begin{equation*}
  \label{eq:Knorm}
  \|T\|_\sEcross := \sup_{S\in\sL^1\cap\sL^\infty, \|S\|_\sE \le 1}
  \tau(TS) < \infty \},
\end{equation*}
see, for example,~\cite{DDP1993}.  It is a subspace of the dual
space~$\sE^*$ (the norms~$\|\cdot\|_{\sEcross}$ and~$\|\cdot\|_{\sE^*}$
coincide on~$\sEcross$) and~$\sEcross = \sE^*$ if and only if the
space~$E$ is separable.  It is known that~$(\sL^p)^\times$, $1\leq p
\leq \infty$ coincides with~$\sL^{p'}$, where~$p'$ is the conjugate
exponent, i.e.\ $p^{-1} + p'^{-1} = 1$.

In the present text, we prove a number of results concerning
perturbation and commutator estimates in noncommutative symmetric
spaces which are relevant to noncommutative geometry.  In this
context, the main interest lies with symmetric spaces~$E \subseteq
L^\infty(0, \infty)$, that is~$\sE$ can be thought of as a unitarily
invariant ideal of~$\aM$ equipped with a unitarily invariant norm.
If~$\aM$ is a type~$I$ factor, then such ideals are customarily called
symmetrically normed ideals (of compact operators), see
e.g.~\cite{GohbergKrein}.

Let~$D_0, D_1 \eta \aM$ be self-adjoint linear operators and let~$a
\in \aM$.  We adopt the following definition,
see~\cite{BraRob1987-MR887100} (see also~\cite{PoSu}).  We shall say
that the operator~$D_0a - aD_1$ is well defined and bounded
(equivalently~$D_0 a - a D_1 \in \sL^\infty$) if and only if
(i)~$a(\Dom{D_1}) \subseteq \Dom{D_0}$; (ii)~the operator~$D_0 a - a
D_1$, initially defined on~$\Dom{D_1}$, is closable; (iii)~the
closure~$\overline{D_0 a - a D_1}$ is bounded.  In this case, the
symbol~$D_0 a - a D_1$ also stands for the closure~$\overline{D_0 a -
  a D_1}$.  In the special case~$D_0 = D_1 = D$, we shall write~$[D,
a] \in \sL^\infty$.

\subsubsection{Double Operator Integrals}
\label{sec:DOI}

Let~$X, Y$ be a normed spaces.  Recall that~$\Bd(X, Y)$ stands for the
normed space of all bounded linear operators~$T: X \mapsto Y$.  If~$X =
Y$, then we shall write~$\Bd(X)$.

Throughout this paper we will let $D_0, D_1$ denote self-adjoint
unbounded operators
affiliated with~$\aM$ and let~$dE_\lambda^0$, $dE_\mu^1$ be the
corresponding spectral measures.  Recall that $$ \tau(x\,
dE_\lambda^0\, y\, dE^1_\mu),\ \ \lambda, \mu \in \Rl $$ is a
$\sigma$-additive complex-valued measure on the plane~$\Rl^2$ with the
total variation bounded by~$\|x\|_{2} \|y\|_{2}$, for every~$x, y \in
\sL^2$, see~\cite[Remark~3.1]{PSW-DOI}.

Let~$\phi = \phi(\lambda, \mu)$ be a bounded Borel function
on~$\Rl^2$.  We call the function~$\phi$ {\it $dE^0 \otimes
  dE^1$-integrable\/} in the space~$\sE$, $1\leq p \leq \infty$ if
and only if there is a linear operator~$T_\phi = T_\phi(D_0, D_1) \in
B(\sE)$ such that
\begin{equation}
  \label{DOIdefFmla}
  \tau (x\, T_\phi(y)) = \int_{\Rl^2}
  \phi(\lambda, \mu)\, \tau(x\, dE_\lambda^0 \, y \, dE_\mu^1),
\end{equation}
for every $$
x \in \sL^2 \cap \sEcross \ \ \text{and}\ \ y \in \sL^2
\cap \sE. $$
If the operator~$T_\phi(D_0, D_1)$ exists, then it is
unique,~\cite[Definition~2.9]{PSW-DOI}.  The latter definition is in
fact a special case of~\cite[Definition~2.9]{PSW-DOI}.  See
also~\cite[Proposition~2.12]{PSW-DOI} and the discussion there on
pages~81--82.  The operator~$T_\phi$ is called {\it the Double Operator
  Integral.}

We shall write~$\phi \in \Phi(\sE)$ if and only if the function~$\phi$
is $dE^0 \otimes dE^1$-integrable in the space~$\sE$ for any
measures~$dE^0$ and~$dE^1$.  The following result is used throughout the
text.

\begin{thm}[{\cite{PSW-DOI, PS-DiffP}}]
  \label{HomomorphismResult}
  Let~$D_0, D_1 \eta \aM$.  The mapping
  $$ \phi \mapsto T_\phi = T_\phi(D_0, D_1) \in \Bd(\sE),\ \ \phi
  \in \Phi(\sE) $$ is a $*$-homomorphism.  Moreover,
  if~$\phi(\lambda, \mu) = \alpha(\lambda)$ (resp.\ $\phi(\lambda,
  \mu) = \beta(\mu)$), $\lambda, \mu \in \Rl$, then $$ T_\phi(x) =
  \alpha(D_0)\, x\ \ (resp.\ T_\phi(x) = x\, \beta(D_1)),\ \ x \in
  \sE, $$ where~$\alpha, \beta: \Rl \mapsto \Cx$ are bounded Borel
  functions.
\end{thm}

The latter result allows the construction of  a sufficiently large class of
functions in~$\Phi(\sE)$, $1\leq p \leq \infty$.  Indeed, let us
consider the class~${\mathfrak{A}_0}$ which consists of all bounded
Borel functions~$\phi(\lambda, \mu)$, $\lambda, \mu \in \Rl$ admitting
the representation
\begin{equation}
  \label{AnotReps}
  \phi(\lambda, \mu) = \int_S \alpha_s(\lambda)\,
  \beta_s(\mu)\, d\nu(s)
\end{equation}
such that $$ \int_S \|\alpha_s\|_\infty \, \|\beta_s\|_\infty \,
d\nu(s) < \infty, $$ where~$(S, d\nu)$ is a measure space, $\alpha_s,
\beta_s:\Rl \mapsto \Cx$ are bounded Borel functions, for every~$s \in
S$ and~$\|\cdot\|_\infty$ is the uniform norm.  The
space~$\mathfrak{A}_0$ is endowed with the norm $$
\|\phi\|_{\mathfrak{A}_0} := \inf \int_S \|\alpha_s\|_\infty \,
\|\beta_s\|_\infty \, d\nu(s), $$ where the minimum runs over all possible
representations~\eqref{AnotReps}.  The space~$\mathfrak{A}_0$ together
with the norm~$\|\cdot\|_{\mathfrak{A}_0}$ is a Banach algebra,
see~\cite{PS-DiffP} for details.  The following result is a
straightforward corollary of Theorem~\ref{HomomorphismResult}.

\begin{corl}[{\cite[Proposition 4.7]{PS-DiffP}}]
  \label{AzeroClass}
  Every~$\phi \in {\mathfrak{A}_0}$ is $dE^0 \otimes dE^1$-integrable in
  the space~$\sE$ for any measures~$dE^0$, $dE^1$, i.e.\
  ${\mathfrak{A}_0} \subseteq \Phi(\sE)$.  Moreover, if\/~$T_\phi =
  T_\phi(D_0, D_1)$, for some self-adjoint operators~$D_0, D_1 \eta
  \aM$, then $$
  \|T_\phi\|_{\Bd(\sE)} \leq \|\phi\|_{\mathfrak{A}_0}, $$
  for every~$\phi \in \mathfrak{A}_0$.
\end{corl}

The following result explains the connection between Double Operator
Integrals and Lipschitz and commutator estimates, see
also~\cite{PS-RFlow}.

\begin{thm}[{\cite[Theorem~3.1]{PoSu}}]
  \label{DOIcommutator}
  Let~$D_0, D_1 \eta \aM$ be self-adjoint linear operators, let~$a \in
  \aM$ and let~$f: \Rl \mapsto \Cx$ be a $C^1$-function with bounded
  derivative.  Let $$ \psi_f(\lambda, \mu) = \frac {f(\lambda) -
    f(\mu)}{\lambda - \mu},\ \ \text{if}\ \lambda \neq \mu
  $$ and~$\psi_f(\lambda, \lambda) = f'(\lambda)$.  If\/~$D_0 a - a
  D_1 \in \sL^\infty$ and~$\psi_f \in \Phi(\sL^\infty)$,
  then~$f(D_0)\, a - a \, f(D_1) \in \sL^\infty$ and $$ f(D_0)\, a - a
  \, f(D_1) = T_{\psi_f} (D_0 a - a D_1), $$ where~$T_{\psi_f} =
  T_{\psi_f} (D_0, D_1)$.
\end{thm}

The result above stated and proved in~\cite[Theorem~3.1]{PoSu} under the
additional assumption that $\aM$ is taken in its left regular
representation. As shown in~\cite[Theorem~2.4.3]{PotapovThesis} this
assumption is redundant.

A decomposition of~$\psi_f$ for the function~$f$
from~\eqref{MainFuncIntro} in the form~\eqref{AnotReps} and further
analysis of this decomposition given in this paper show that~$T_{\psi_f}
\in \Bd(\sL^\infty, \sE)$ for every symmetric space~$\sE$ and this
result underlies the applications of double operator integration theory
to noncommutative geometry given in Section~\ref{FredholmSec}.  In the
rest of this section, we collect some preliminary material for this
analysis.

Recall that~$\Lambda_\alpha$, $0\leq \alpha \leq 1$ stands for the
class of all H\"older functions, i.e. the functions~$f: \Rl \mapsto
\Cx$ such that $$ \|f\|_{\Lambda_\alpha} := \sup_{t_1, t_2} \frac
{|f(t_1) - f(t_2)|}{|t_1 - t_2|^\alpha} < + \infty. $$

\begin{thm}
  \label{HolderCriterion}
  Let~$f: \Rl \mapsto \Cx$.  If\/~$\|f\|_{\Lambda_\theta},
  \|f'\|_\infty, \|f'\|_{\Lambda_\epsilon} < \infty$, for some~$0 \leq
  \theta < 1$ and~$0 < \epsilon \leq 1$, then $\psi_f \in
  \mathfrak{A}_0$.  Moreover, there is a constant~$c = c(\epsilon,
  \theta)>0$ such that $$ \|\psi_f\|_{\mathfrak{A}_0} \leq c\,
  (\|f\|_{\Lambda_\theta} + \|f'\|_\infty + \|f'\|_{\Lambda_\epsilon}).
  $$
\end{thm}

\begin{proof}
  We let symbol~$c$ stand for a positive constant which may vary from
  line to line.  The proof is based on the following result due to
  V.Peller~\cite{Peller1985-MR800919}, see also~\cite{PS-DiffP}.  If~$f'
  \in L^\infty$ and~$f \in \dot B^1_{\infty, 1}$, then~$\psi_f \in
  {\mathfrak{A}_0}$ and
  \begin{equation*}
    \|\psi_f\|_{\mathfrak{A}_0} \leq c\,
    (\|f'\|_\infty + \|f\|_{\dot B^1_{\infty, 1}}),
  \end{equation*}
  where~${\dot B^1_{\infty, 1}}$ is the homogeneous Besov class,
  see~\cite{Petre1976-MR0461123, Stein1970-MR0290095}.  To finish the
  proof, we shall show that
  \begin{equation}
    \label{HolderCriterionTmp}
    \|f\|_{\dot B^1_{\infty, 1}} \leq c\,
    (\|f\|_{\Lambda_\theta} + \|f'\|_{\Lambda_\epsilon}).
  \end{equation}
  The argument is rather standard.  Let~$f(t)$ be a function and
  let~$u(t, s)$, $s > 0$ be the Poisson integral of the function~$f$,
  i.e. $$ u(t, s) = f * P_s(t) = \int_{\Rl} f(\tau)\, P_s(t - \tau)\,
  d\tau,\ \ t\in \Rl,\ s > 0, $$ where~$P_s(t)$ is the Poisson kernel,
  i.e. $$ P_s(t) = \frac 1 \pi\, \frac s {t^2 + s^2}. $$ Let~$u'_s$
  and~$u''_{ss}$ stand for the derivatives~$\frac {\partial u}{\partial
    s}$ and~$\frac {\partial^2 u} {\partial s^2}$, respectively.  Recall
  that, for every~$0 \leq \alpha \leq 1$, there is a numerical
  constant~$c_\alpha > 0$ such that (see~\cite[Ch.~V,
  Section~4.2]{Stein1970-MR0290095})
  \begin{equation}
    \label{HolderCriterionTemp}
    \sup_{s > 0} s^{1 - \alpha} \, \|u'_s\|_\infty \leq \, c_\alpha \,
    \|f\|_{\Lambda_\alpha}.
  \end{equation}
  Recall also that (see~\cite{Stein1970-MR0290095, Petre1976-MR0461123})
  \begin{equation}
    \label{HolderCriterionTempII}
    \|f\|_{\dot B^1_{\infty, 1}} \sim \int_0^\infty
    \|u''_{ss}\|_\infty\, ds
  \end{equation}
  with equivalence up to a positive numerical constant.

  Fix~$f$ such that~$f \in \Lambda_\theta$, $0 \leq \theta < 1$ and~$f'
  \in \Lambda_\epsilon$, $0 < \epsilon \leq 1$.  It now follows
  from~\eqref{HolderCriterionTemp} that
  \begin{equation}
    \label{HolderCriterionTempIV}
    \|u'_s\|_\infty \leq c_\theta\, \frac {\|f\|_{\Lambda_\theta}}{s^{1 - \theta}}
    \ \ \text{and}\ \ \|u''_{ss}\|_\infty \leq c_\epsilon\, \frac
    {\|f'\|_{\Lambda_\epsilon}} {s^{1 - \epsilon}}.
  \end{equation}
  The Poisson kernel~$P_s$ possesses the group property~$P_{s_1} *
  P_{s_2} = P_{s_1 + s_2}$, $s_1, s_2 > 0$.  Consequently, $$
  u(s_1 +
  s_2, t) = u(s_1, \cdot) * P_{s_2} (t). $$
  Taking~$\frac{\partial^2}{\partial s_1 \partial s_2}$ and then
  letting~$s_1 = s_2 = \frac s 2$ yields $$
  u''_{ss} (s, t) = u'_s(s/2,
  \cdot) * \frac {\partial P_{s/2}}{\partial s} (t). $$
  The latter
  implies
  \begin{equation}
    \label{HolderCriterionTempIII}
    \|u''_{ss}\|_\infty \leq \|u'_s\|_\infty\, \left\| \frac {\partial
        P_{s/2}}{\partial s}  \right\|_1 \leq c_0\, \frac
    {\|u'_s\|_\infty}s,\ \ s > 0,
  \end{equation}
  where~$c_0$ is given by $$ c_0 = s\, \left\| \frac{\partial
      P_s}{\partial s}\right\|_1 > 0. $$
  Combining~\eqref{HolderCriterionTempIII} with the first estimate
  in~\eqref{HolderCriterionTempIV} yields $$ \|u''_{ss}\|_\infty \leq
  c\, \frac {\|f\|_{\Lambda_\theta}}{s^{2 -\theta}}. $$ Now the last
  inequality together with the second estimate
  in~\eqref{HolderCriterionTempIV} gives
  \begin{align*}
    \|f\|_{\dot B^1_{\infty, 1}} \leq &\, c\, \int_0^\infty
    \|u''_{ss}\|_\infty\, ds \cr = &\, c \, \int_0^1 \|u''_{ss}\|_\infty
    \, ds + c\, \int_1^\infty \|u''_{ss}\|_\infty\, ds \cr \leq &\, c\,
    \|f'\|_{\Lambda_\epsilon}\, \int_0^1 \frac {ds}{s^{1 - \epsilon}} +
    c\, \|f\|_{\Lambda_\theta} \, \int_1^\infty \frac {ds}{s^{2 -
        \theta}} \cr \leq &\, c\, \left(\strut
      \|f'\|_{\Lambda_\epsilon} + \|f\|_{\Lambda_\theta}\right).
  \end{align*}
  The latter finishes the proof of~\eqref{HolderCriterionTmp}.  The
  theorem is proved.
\end{proof}

\begin{rem}
  Theorem~\ref{HolderCriterion} is stated in a rather restrictive form
  since the requirement~$\|f'\|_\infty < \infty$ is redundant.  Indeed,
  it may be shown that $$ \|f'\|_\infty \leq c(\theta, \epsilon) \,
  (\|f\|_{\Lambda_\theta} + \|f'\|_{\Lambda_\epsilon}),\ \ 0\leq \theta
  \leq 1,\ 0 \leq \epsilon \leq 1. $$ On the other hand, for all
  applications of Theorem~\ref{HolderCriterion} in the text below the
  requirement~$\|f'\|_\infty < \infty$ clearly holds.
\end{rem}

The following is a well-known criterion to verify boundedness of the
operator~$T_\phi = T_{\phi}(D_0, D_1)$.  We supply a simple proof for
convenience of the reader.

\begin{lem}
  \label{DOIbddCr}
  Let~$D_0, D_1 \eta \aM$ be self-adjoint linear operators, let~$f:
  \Rl \mapsto \Cx$ and let~$\hat f$ be the Fourier transform of~$f$.
  If~$\hat f$ is integrable, i.e.~$\hat f \in L^1(\Rl)$,
  then~$T_{\phi} = T_\phi(D_0, D_1) \in \Bd(\sL^\infty)$,
  where~$\phi(\lambda, \mu) = f(\lambda - \mu)$, $\lambda, \mu \in
  \Rl$ and $$ \|T_\phi\|_{\Bd(\sL^\infty)} \leq \frac 1{\sqrt{2\pi}}\,
  \|\hat f\|_{L^1}. $$
\end{lem}

\begin{proof}
  The proof is straightforward.  For the function~$\phi(\lambda, \mu)$
  we have the representation $$ \phi(\lambda, \mu) = f(\lambda - \mu)
  = \frac 1{\sqrt{2 \pi}} \int_{\Rl} \hat f(s)\, e^{-is(\lambda -
    \mu)}\, ds. $$ Since~$\hat f$ is integrable, we readily obtain
  that~$\phi \in \mathfrak{A}_0$ and~$\|\phi\|_{\mathfrak{A}_0} \leq
  \frac 1{\sqrt {2 \pi}} \, \|\hat f \|_{L^1}$. The claim of the lemma
  now follows from Corollary~\ref{AzeroClass}.
\end{proof}

Note, that the norm estimate of the operator~$T_\phi = T_\phi(D_0, D_1)$
in the latter lemma does not depend on the operators~$D_0$ and~$D_1$.
Next we shall give a simple criterion (from~\cite{BraRob1987-MR887100})
for a Borel function~$f: \Rl \mapsto \Cx$ to be such that~$\hat f \in
L^1(\Rl)$.  We shall present the proof for convenience of the reader.

\begin{lem}
  \label{FourierLi}
  If~$f: \Rl \mapsto \Cx$ is an absolutely continuous function with~$f,
  f' \in L^2(\Rl)$, then~$\hat f \in L^1(\Rl)$ and $$ \left\| \hat f
  \right\|_{L^1} \leq \sqrt 2 \, \left( \left\| f \right\|_{L^2} +
    \left\| f' \right\|_{L^2} \right). $$
\end{lem}

\begin{proof}
  The proof is a combination of the H\"older inequality and the
  Plancherel identity
  \begin{align*}
    \int_{\Rl} | \hat f(t)|\, dt = &\, \int_{t \in [-1, 1]} |\hat
    f(t)|\, dt + \int_{t \not \in [-1, 1]} |t|^{-1} \, |t \hat f(t)|
    \, dt \cr \leq &\, \sqrt 2\, \left [ \int_{t \in [-1, 1]} |\hat
      f(t)|^2 \,dt \right ]^{\frac 12} \cr + &\, \left [ \int_{t \not
        \in [-1, 1]} |t|^{-2}\, dt \right ]^{\frac 12} \cdot \left [
      \int_{t \not \in [-1, 1]} |t \hat f(t)|^2\, dt \right ]^{\frac
      12} \cr \leq &\, \sqrt 2\, \left( \|f\|_{L^2} + \|f'\|_{L^2} \right)
  \end{align*}
\end{proof}

Let~$(S, \mathfrak{F})$ and~$(S', \mathfrak{F}')$ be two measure spaces
and let~$\nu$ be a measure on~$(S, \mathfrak{F})$.  If~$\omega : (S,
\mathfrak{F}) \mapsto (S', \mathfrak{F}')$ is a measurable mapping,
i.e.\ $\omega: S \mapsto S'$ and~$\omega^{-1}(A) \in \mathfrak{F}$ for
every~$A \in \mathfrak{F}'$, then the mapping~$\omega$ induces the
measure~$\nu \circ \omega^{-1}$ on the space~$(S', \mathfrak{F}')$ by
the assigning $$ \nu \circ \omega^{-1}(A) := \nu ( \omega^{-1}(A)),\ \ A
\in \mathfrak{F}'. $$ If~$f: S' \mapsto \Cx$ is a
$\mathfrak{F}'$-measurable function, then~$f \circ \omega$ is
$\mathfrak{F}$-measurable and
\begin{equation}
  \label{ClassicalExchange}
  \int_S f \circ \omega\, d \nu = \int_{S'} f \, d \nu \circ \omega^{-1},
\end{equation}
provided either of the of the integrals exists.  The following lemma
extends this relation to the setting of double operator integrals.

\begin{lem}
  \label{ChangeOfVarsDOI}
  Let~$\phi \in \mathfrak{A}_0$ and let~$f_j : \Rl \mapsto \Rl$,
  $j=0,1$ be Borel functions.  We have that
  \begin{equation}
    \label{ChangeOfVars}
    T_{\phi'}(D_0, D_1) = T_{\phi}(D_0', D_1'),
  \end{equation}
  where~$\phi' \in \mathfrak{A}_0$, $\phi'(\lambda, \mu) :=
  \phi(f_0(\lambda), f_1(\mu))$, $\lambda, \mu \in \Rl$ and~$D_j' :=
  f_j(D_j) \eta \aM$, $j=0,1$.
\end{lem}

\begin{proof}
  We fix~$x, y \in \sL^2$ and set $$ d \nu = d \nu_{\lambda, \mu} = d
  \nu_{\lambda, \mu} (x, y, D_0, D_1) := \tau(x\, dE^0_\lambda\, y\,
  dE^1_\mu). $$ Let~$T_{\phi'} = T_{\phi'} (D_0, D_1)$.  Consider the
  mapping~$\omega : \Rl^2 \mapsto \Rl^2$ given by $$ (\lambda, \mu)
  \mapsto (f_0(\lambda), f_1(\mu)). $$ Note that~$\phi' = \phi \circ
  \omega$.  Applying identity~\eqref{ClassicalExchange}
  and~\eqref{DOIdefFmla}, we now have
  \begin{equation}
    \label{ChangeOfVarTempI}
    \tau(x\, T_{\phi'}(y)) = \int_{\Rl^2} \phi'\, d
    \nu  = \int_{\Rl^2} \phi \circ \omega \,
    d \nu = \int_{\Rl^2} \phi \, d \nu \circ
    \omega^{-1}.
  \end{equation}
  The measure~$\nu \circ \omega^{-1}$ is given by
  \begin{equation}
    \label{ChangeOfVarTempII}
    \nu \circ \omega^{-1} = \tau (x \, E^0\, y \, E^1 )
    \circ \omega^{-1} = \tau (x \, (E^0 \circ
    f_0^{-1}) \, y \, (E^1 \circ f_1^{-1})),
  \end{equation}
  where the spectral measure~$E^j \circ f_j^{-1}$, $j=0,1$ is defined
  by $$
  E^{j} \circ f_j^{-1}(A) := E(f^{-1}(A)),\ \ j=0,1, $$
  for Borel
  set~$A \subseteq \Rl$.  Let us note that,
  applying~\eqref{ClassicalExchange} again (see
  also~\cite[Section~13.28]{Rudin1991-MR1157815}), we see that $$
  \int_{\Rl} \lambda \, d (E^j \circ f_j^{-1})_\lambda = \int_{\Rl}
  f_j(\lambda) \, dE^j_\lambda = f_j(D_j),\ \ j=0,1. $$
  Thus, the
  measure~$dF^j := d(E^j \circ f_j^{-1})$ is the spectral measure of the
  operator~$f_j(D_j)$, $j = 0,1$.  Consequently,
  combining~\eqref{ChangeOfVarTempI} and~\eqref{ChangeOfVarTempII}, we
  readily obtain that $$
  \tau(x\, T_{\phi'}(y)) = \int_{\Rl^2}
  \phi(\lambda, \mu)\, d \tau(x\, dF^0_\lambda \, y\, dF^1_\mu) =
  \tau(x\, T_{\phi} (y)), $$
  where~$T_\phi = T_\phi(D_0', D_1')$.  The
  latter identity, together with uniqueness of the operator~$T_\phi$
  satisfying~\eqref{DOIdefFmla}, completes the proof of the lemma.
\end{proof}

The identity~\eqref{ChangeOfVars} together with Lemma~\ref{DOIbddCr}
yield

\begin{lem}
  \label{DOIbddCrCor}
  Let~$D_0, D_1 \eta \aM$ be positive linear operators, let~$f: \Rl
  \mapsto \Cx$ be Borel and let~$g(t) = f(e^{t})$, $t \in \Rl$.
  If\/~$\hat g \in L^1(\Rl)$, then\/~$T_{\phi}(D_0, D_1) \in
  \Bd(\sL^\infty)$, where $$ \phi(\lambda, \mu) = f \left (\frac
    \lambda \mu \right ),\ \ \lambda, \mu > 0 $$ and $$
  \|T_\phi\|_{\Bd(\sL^\infty)} \leq \frac 1{\sqrt{2 \pi}}\, \|\hat
  g\|_{L^1}. $$ Furthermore, the decomposition~\eqref{AnotReps} for
  the function~$\phi(\lambda, \mu)$ is given by $$ \phi(\lambda, \mu)
  = \int_{\Rl} \hat g(s)\, \lambda^{is}\, \mu^{-is}\, ds,\ \ \lambda,
  \mu > 0. $$

\end{lem}

\begin{proof}
  Let us introduce the operator~$D'_j := \log D_j$, $j=0,1$ and the
  function~$\phi'(\lambda, \mu) = \phi(e^\lambda, e^\mu) =
  f(e^{\lambda - \mu}) = g(\lambda - \mu)$, $\lambda, \mu \in \Rl$.
  Note, that~$D'_j \eta\aM$, $j=0,1$.  Since~$\hat g \in L^1(\Rl)$, it
  readily follows from Lemma~\ref{DOIbddCr} that~$T_{\phi'}(D'_0,
  D'_1) \in \Bd(\sL^\infty)$.  On the other hand,
  from~\eqref{ChangeOfVars}, we obtain that~$T_\phi(D_0, D_1) =
  T_{\phi'}(D_0', D_1') \in \Bd(\sL^\infty)$.  Furthermore, it follows
  from Lemma~\ref{DOIbddCr} that the function~$\phi'$ has the
  decomposition $$ \phi(e^\lambda, e^\mu) = \phi'(\lambda, \mu) =
  \int_{\Rl} \hat g(s)\, e^{is (\lambda - \mu)} \, ds. $$ Making the
  back substitution finishes the proof of the lemma.
\end{proof}

At the end of the section we prove the following lemma which is
implicit in literature and frequently used in the following section.

\begin{lem}
  \label{MaximumPrincCorollary}
  If~$A \in \aM$ and if~$B_0, B_1 \in \sE$ and~$B_0, B_1$ are positive,
  then~$B_0^{1- \theta} A B^{\theta}_1 \in \sE$, for every~$0\leq \theta
  \leq 1$ and $$
  \|B_0^{1-\theta} A B_1^\theta\|_\sE \leq \|B_0\|_\sE^{1
    - \theta} \|A\|\, \|B_1\|_\sE^\theta. $$
\end{lem}

\begin{proof}
  Consider the following holomorphic function with values in~$\sE$ $$
  f(z) = \|B_0\|_\sE^{z - 1} \|B_1\|_\sE^{-z} B_0^{1 - z} A B_1^{z},\
  \ z \in \Cx. $$ Clearly, we have $$ \sup_{t \in \Rl}
  \|f(it)\|_{\sE},\ \sup_{t \in \Rl} \|f(1 + it)\|_{\sE} \leq \|A\|.
  $$ Since the function~$f$ is holomorphic, the claim of the lemma
  follows from the maximum principle applied to the strip~$0\leq \Re z
  \leq 1$.
\end{proof}

\subsubsection{Lipschitz and commutator estimates.}
\label{BasicSec}

The objective of this Section is to establish a general approach to
proving the kind of commutator estimates that arise in noncommutative
geometry.
In the present section we fix self-adjoint linear operators~$D_0, D
\eta\aM$ and, for every~$\alpha \geq 0$, we set $$ \Delta_\alpha :=
(\alpha^2 + D^2)^{\frac 12} \ \ \text{and}\ \ \Delta_{0, \alpha} :=
(\alpha^2 + D_0^2)^{\frac 12} $$ and $$ \Delta := (1 + D^2)^{\frac 12}
\ \ \text{and}\ \ \Delta_{0} := (1 + D_0^2)^{\frac 12}. $$

\begin{thm}
  \label{CommResult}
  Let~$\alpha > 0$, $\Delta_\alpha^{-1} \in \sE$ and~$a \in \aM$.
  If\/~$[D, a] \in \sL^\infty$, then~$[D \Delta_\alpha^{-1}, a] \in \sE$
  and there is a constant~$c>0$ independent of~$\alpha$ such that
  $$ \| [ D \Delta_\alpha^{-1}, a] \|_\sE \leq c\, \left\|
    \Delta_\alpha^{-1} \right\|_{\sE} \|[D, a]\|. $$
\end{thm}

\begin{proof}
  Let us consider the functions
  \begin{equation}
    \label{FuncFalpha}
    f_\alpha(t) = \frac t{(\alpha^2 + t^2)^{\frac 12}},\ t\in \Rl \ \
    \text{and}\ \  \psi_\alpha = \psi_{f_\alpha}.
  \end{equation}
  By Theorem~\ref{HolderCriterion} and Corollary~\ref{AzeroClass}, we
  have~$\psi_\alpha = \psi_{f_\alpha} \in \Phi(\sL^\infty)$, and
  therefore~$T_\alpha = T_{\psi_\alpha}(D, D) \in \Bd(\sL^\infty)$.
  Consequently, it follows from Theorem~\ref{DOIcommutator} that $$ [D
  \Delta_\alpha^{-1}, a] = T_\alpha ([D, a]). $$ We shall show
  that
  \begin{equation}
    \label{CommResultClaimI}
    T_\alpha \in \Bd(\sL^\infty, \sE).
  \end{equation}
  Consider the following representation of the
  function~$\psi_\alpha$.
  \begin{align}
    \label{CommResultRep}
    \psi_\alpha(\lambda, \mu) = &\, \frac{ \lambda (\alpha^2 +
      \lambda^2)^{-\frac 12} - \mu (\alpha^2 + \mu^2)^{-\frac 12} }{
      \lambda - \mu} \cr = &\, \frac {(\lambda + \mu)\, (\lambda
      (\alpha^2 + \lambda^2)^{-\frac 12} - \mu (\alpha^2 +
      \mu^2)^{-\frac 12})}{(\alpha^2 + \lambda^2) - (\alpha^2 + \mu^2)}
    \cr = &\, \frac{ (\alpha^2 + \lambda^2)^{\frac 12} + (\lambda \mu -
      \alpha^2 )(\alpha^2 +\lambda^2)^{\frac 12}} {(\alpha^2 +\lambda^2)
      - (\alpha^2 + \mu^2)} \cr - &\, \frac { (\alpha^2 +\mu^2)^{\frac
        12} + (\lambda \mu -\alpha^2 ) (\alpha^2 +\mu^2)^{- \frac 12}}{
      (\alpha^2 +\lambda^2) - (\alpha^2+\mu^2)} \cr = &\,
    \psi_\alpha'(\lambda, \mu)\, \left ( 1 + \frac{\alpha^2 - \lambda
        \mu}{(\alpha^2+\lambda^2)^{\frac 12} (\alpha^2 + \mu^2)^{\frac
          12}} \right),
  \end{align}
  where
  \begin{equation}
    \label{psiPrimeAlpha}
    \psi'_\alpha(\lambda, \mu) := \frac 1
    {(\alpha^2+\lambda^2)^{\frac 12} + (\alpha^2+ \mu^2)^{\frac 12}}.
  \end{equation}
  The corresponding resolution for~$T_\alpha$ is given by
  \begin{equation}
    \label{TalphaDecomp}
    T_\alpha
    (x) = T_\alpha'(x) + \alpha^2 \Delta_\alpha^{-1} T_\alpha'(x) \,
    \Delta_\alpha^{-1} - D \Delta_\alpha^{-1} T_\alpha'(x)\, D
    \Delta_\alpha^{-1},\ \ x \in \sL^\infty,
  \end{equation}
  where we have put for brevity~$T_\alpha' = T_{\psi'_\alpha} (D, D)$.
  Note that~$\alpha \|\Delta_\alpha^{-1}\|, \|D \Delta_\alpha^{-1}\|
  \leq 1$ for every~$\alpha > 0$.  Thus, the
  claim~\eqref{CommResultClaimI} will readily follow as soon as we
  establish that
  \begin{equation}
    \label{CommResultClaimII}
    T'_\alpha \in \Bd(\sL^\infty, \sE).
  \end{equation}
  By Lemma~\ref{ChangeOfVarsDOI}, we have
  \begin{equation}
    \label{CommResultClaimIITemp}
    T'_\alpha
    = T_{\psi'_\alpha}(D, D) = T_{\psi_0}(\Delta_\alpha, \Delta_\alpha) =:
    T_0,
  \end{equation}
  where $$ \psi_0(\lambda, \mu) = \frac{1}{\lambda + \mu},\ \ \lambda,
  \mu >0. $$ Therefore, it is sufficient to show that
  \begin{equation}
    \label{CommResultClaimIII}
    T_0 \in \Bd(\sL^\infty, \sE).
  \end{equation}
  Representing the function~$\psi_0$ as
  \begin{equation}
    \label{psiNotRepr}
    \psi_0(\lambda, \mu)
    = \frac 1 {\lambda + \mu} = \lambda^{- \frac 12} \mu^{- \frac 12}
    \frac 1 { \left (\frac \lambda \mu \right)^{\frac 12} + \left ( \frac
        \lambda \mu \right)^{- \frac 12}} =: \lambda^{- \frac 12} \mu^{-
      \frac 12} \psi(\lambda, \mu),
  \end{equation}
  and setting~$T_\psi := T_{\psi}(\Delta_\alpha, \Delta_\alpha)$, we
  obtain
  \begin{equation}
    \label{CommResultLastTmp}
    T_0(x) =
    \Delta_\alpha^{- \frac 12} T_\psi(x) \, \Delta_\alpha^{- \frac 12},\
    \ x \in \sL^\infty.
  \end{equation}
  Note that in~\eqref{CommResultLastTmp}, we have used
  Theorem~\ref{HomomorphismResult} which is applicable here since the
  functions~$\lambda^{-\frac 12}$ and~$\mu^{- \frac 12}$ are bounded
  on the spectrum of the operator~$\Delta_\alpha$.  By the assumption,
  $\Delta_\alpha^{-1} \in \sE$ and
  therefore,~\eqref{CommResultClaimIII} follows
  from~\eqref{CommResultLastTmp} via Lemma~\ref{MaximumPrincCorollary}
  provided we know that~$T_\psi \in \Bd(\sL^\infty)$.  The definition
  of the function~$\psi$ (see~\eqref{psiNotRepr}) and
  Lemma~\ref{DOIbddCrCor} guarantee the latter embedding provided
  that~$\hat g \in L^1$, where the function~$g$ is given by
  \begin{equation}
    \label{TheFirstGfunc}
    g(t) = \frac 1{e^{\frac t2} + e^{-
        \frac t 2}},\ \ t \in \Rl.
  \end{equation}
  To see that~$\hat g \in L^1$ it is sufficient to observe that~$g, g'
  \in L^2$ and apply Lemma~\ref{FourierLi}.
\end{proof}

\begin{rem}
  For future use, we note that by Lemma~\ref{DOIbddCrCor}, the
  functions~$\psi$ and~$g$ defined in~\eqref{psiNotRepr}
  and~\eqref{TheFirstGfunc} respectively satisfy the equality
  \begin{equation}
    \label{psiDecomp}
    \psi(\lambda, \mu) = \int_{\Rl} \hat g(s)\, \lambda^{is}\,
    \mu^{-is}\, ds,\ \ \lambda, \mu > 0.
  \end{equation}
\end{rem}

Since the estimate in Theorem~\ref{CommResult} is uniform with respect
to~$\alpha > 0$, letting~$\alpha \rightarrow 0$ and noting that the left
hand side of the estimate in Theorem~\ref{CommResult} tends to~$[\sgn
D, a]$ in the weak operator topology proves

\begin{corl}
  \label{CommResultCor}
  Let~$1\leq p \leq \infty$, $|D|^{-1} \in \sL^p$ and~$a \in \aM$.
  If\/~$[D, a] \in \sL^\infty$, then~$[\sgn D, a] \in \sL^p$ and there
  is a constant~$c > 0$ such that
  $$ \|[\sgn D, a]\|_p \leq c\, \|\,|D|^{-1}\|_p\, \|[D, a]\|. $$
\end{corl}

The latter result was proved for the setting of~$\aM = \Bd(\sH)$ by
elementary (and different) reasoning in~\cite{ScWaWa1998}
(see also the exposition of this result
in~\cite[Lemma~10.18]{GraVar2001-MR1789831}).

Let us note that the argument in the proof of Theorem~\ref{CommResult}
also works for Lipschitz estimates.  The precise statement follows.

\begin{thm}
  \label{QcommResultWeak}
  Let~$\alpha \geq 0$, $\Delta_{0, \alpha}^{-1}, \Delta_\alpha^{-1} \in
  \sE$.  If\/~$D - D_0 \in \sL^\infty$, then~$D \Delta_\alpha^{-1} - D_0
  \Delta_{0, \alpha}^{-1} \in \sE$ and there is a constant~$c > 0$
  independent of~$\alpha$ such that, for every~$0 < \theta < 1$, $$ \| D
  \Delta_\alpha^{-1} - D_0 \Delta_{0, \alpha}^{-1} \|_\sE \leq \,
  c_\theta\, \|\Delta^{-1}_{0, \alpha}\|^{1-\theta}_\sE \,
  \|\Delta_\alpha^{-1}\|^\theta_\sE\, \|D - D_0\|, $$ where
  \begin{equation}
    \label{cthetaEst}
    c_\theta \leq c\, \max \left\{ \theta^{- \frac 12}, (1 -
      \theta)^{- \frac 12} \right \}.
  \end{equation}
\end{thm}

\begin{proof}
  The proof is a repetition of that of Theorem~\ref{CommResult}.  The
  only place which requires additional reasoning is the
  estimate~\eqref{cthetaEst}.  To this end, we shall estimate the norm
  of the operator~$T_0$ from~\eqref{CommResultClaimIITemp}
  differently.  We slightly modify representation~\eqref{psiNotRepr}
  $$ \psi_0(\lambda, \mu) = \frac 1 {\lambda + \mu} = \lambda^{\theta
    - 1} \, \mu^{- \theta}\, \frac 1 { \left(\frac \lambda \mu \right
    )^{\theta} + \left ( \frac \lambda \mu \right )^{\theta - 1} } =:
  \lambda ^{\theta - 1}\, \mu^{-\theta} \, \psi_\theta(\lambda, \mu).
  $$ Note that $$ T_0(x) = \Delta_{0, \alpha}^{\theta - 1}\,
  T_\theta(x) \, \Delta_{\alpha}^{- \theta}, $$ where~$T_\theta =
  T_{\psi_\theta}(\Delta_{0,\alpha}, \Delta_{\alpha})$.  Suppose, that
  we know that~$T_\theta \in \Bd(\sL^\infty)$, then
  Lemma~\ref{MaximumPrincCorollary} yields the implication $$ T_0 \in
  \Bd(\sL^\infty, \sE) \ \ \Longleftarrow \ \ T_\theta \in
  \Bd(\sL^\infty) $$ and $$ c_\theta \leq
  \|T_\theta\|_{\Bd(\sL^\infty)}. $$ Setting $$ g(t) = \frac 1
  {e^{\theta t} + e^{(\theta - 1) t}}, t \in \Rl $$ we have $$
  \|g\|_{L^2} + \|g'\|_{L^2} \leq c\, \max \left\{ \theta^{- \frac 12},
    (1 -\theta)^{- \frac 12} \right\}, $$ for some numerical constant~$c
  > 0$, and therefore, by Lemmas~\ref{DOIbddCrCor} and~\ref{FourierLi}
  we indeed have~$T_\theta \in \Bd(\sL^\infty)$ and $$
  \|T_\theta\|_{\Bd(\sL^\infty)} \leq \frac c{\sqrt \pi}\, \max \left\{
    \theta^{- \frac 12}, (1 -\theta)^{- \frac 12} \right\}, $$ which
  yields~\eqref{cthetaEst}.
\end{proof}

Results given in Theorems~\ref{CommResult} and~\ref{QcommResultWeak}
are based on the analysis of the function~$\psi_{f_\alpha}$,
where~$f_\alpha$ is given in~\eqref{FuncFalpha}.  A similar analysis
can be also performed for the function~$\psi_{h_\alpha}$, where $$
h_\alpha(t) = \frac 1 {(\alpha^2 + t^2)^{\frac 12}},\ \ t \in \Rl. $$

\begin{thm}
  \label{InverseResult}
  Let~$\alpha > 0$, $\Delta_{0, \alpha}^{-1}, \Delta_{\alpha}^{-1} \in
  \sE$.  If\/~$D - D_0 \in \sL^\infty$, then~$\Delta_\alpha^{-1} -
  \Delta_{0, \alpha}^{-1} \in \sE$ and, for every~$0 < \theta < 1$, $$
  \|\Delta_\alpha^{-1} - \Delta_{0, \alpha}^{-1}\|_\sE \leq c_\theta\,
  \alpha^{-1}\, \|\Delta^{-1}_{0, \alpha}\|^{1-\theta}_\sE \,
  \|\Delta_\alpha^{-1}\|^\theta_\sE \, \|D - D_0\|, $$ where~$c_\theta$
  satisfies~\eqref{cthetaEst} with some constant~$c > 0$ independent
  of~$\alpha$ and~$\theta$.
\end{thm}

\begin{proof}
  We have
  \begin{equation}
    \label{InverseResultTemp}
    \psi_{h_\alpha}(\lambda, \mu) = \frac {\lambda + \mu}{(\alpha^2 +
      \lambda^2)^{\frac 12} \, (\alpha^2 + \mu^2)^{\frac 12}} \,
    \psi'_\alpha (\lambda, \mu),
  \end{equation}
  where~$\psi'_\alpha$ is given in~\eqref{psiPrimeAlpha}.  Note that
  $$ \left| \frac {\lambda + \mu} { (\alpha^2 + \lambda^2)^{\frac
        12}\, (\alpha^2 + \mu^2)^{\frac 12}} \right| \leq
  \alpha^{-1},\ \ \lambda, \mu \in \Rl $$ and therefore, by
  Theorem~\ref{HomomorphismResult}, $$
  \|T_{\psi_{h_\alpha}}\|_{\Bd(\sL^\infty, \sE)} \leq \alpha^{-1}\,
  \|T_{\psi'_\alpha}\|_{\Bd(\sL^\infty, \sE)}. $$ The claim now
  follows from~\eqref{CommResultClaimII} and the proof of
  Theorem~\ref{QcommResultWeak}.
\end{proof}

Theorems~\ref{QcommResultWeak} and~\ref{InverseResult} require that
~$\|\Delta_{0, \alpha}^{-1}\|_\sE, \|\Delta_{\alpha}^{-1}\|_\sE < +
\infty$.  We shall next relax this hypothesis.

\begin{thm}
  \label{InverseResultWeakened}
  Let~$\alpha > 0$, $\Delta^{-1}_{0, \alpha} \in \sE$.  If\/~$D - D_0
  \in \sL^\infty$, then~$\Delta_\alpha^{-1} - \Delta_{0, \alpha}^{-1}
  \in \sE$ and there is a constnat~$c > 0$ independent of~$\alpha$ such
  that $$ \| \Delta_\alpha^{-1} - \Delta_{0, \alpha}^{-1}\|_\sE \leq c
  \, \max \{1, \alpha^{-1}\}\, \|\Delta_{0, \alpha}^{-1}\|_\sE \, \|D -
  D_0\|. $$
\end{thm}

\begin{proof}
  Let us first assume that~$\|D - D_0\| \leq 1$.  We set $$ A :=
  \frac {\|\Delta_\alpha^{-1} - \Delta_{0, \alpha}^{-1}\|_\sE}
  {\|\Delta_{0, \alpha}^{-1}\|_\sE}. $$ It follows from
  Theorem~\ref{InverseResult} that
  \begin{equation}
    \label{InverseResultWeakenedTemp}
    \|\Delta_\alpha^{-1} -
    \Delta_{0, \alpha}^{-1}\|_\sE \leq c_\theta\, \alpha^{-1}\, \|\Delta_{0,
      \alpha}^{-1}\|^{1-\theta}_\sE \, \|\Delta_\alpha^{-1}\|_\sE^\theta \,
    \|D - D_0\|.
  \end{equation}
  On the other hand, it follows from triangle inequality that
  \begin{equation}
    \label{InverseResultWeakenedTempI}
    \|\Delta_\alpha^{-1}\|_\sE \leq \|\Delta_{0, \alpha}^{-1}\|_\sE
    + \|\Delta_\alpha^{-1} - \Delta_{0, \alpha}^{-1}\|_\sE.
  \end{equation}
  Replacing~$\|\Delta_\alpha^{-1}\|_\sE$ on the right
  in~\eqref{InverseResultWeakenedTemp} with the right-hand side
  of~\eqref{InverseResultWeakenedTempI} and applying the following
  standard inequality $$ (1 + x)^\theta \leq 1 + \theta \, x,\ \
  \theta \leq 1,\ x \geq 0 $$ yields that
  \begin{equation}
    \label{InverseResultWeakenedTempII}
    A \leq c_\theta \, \alpha^{-1}\, \|D - D_0\|\, (1 + \theta A).
  \end{equation}
  Fix~$\theta = \min \left\{\frac 14, \frac {\alpha^2}4\right\}$.
  Since~$\|D - D_0\| \leq 1$ it readily follows from~\eqref{cthetaEst}
  that $$ c_\theta \theta \, \alpha^{-1}\, \|D - D_0\| \leq \frac 12.
  $$ We let~$c_\alpha = \frac {c_\theta}2$.  It follows
  from~\eqref{cthetaEst} that $$ c_\alpha \leq c\, \max \left\{ 1,
    \alpha^{-1}\right\},\ \text{for some~$c > 0$}. $$ It is now clear
  that~\eqref{InverseResultWeakenedTempII} implies
  $$ \frac {\|\Delta_\alpha^{-1} - \Delta_{0, \alpha}^{-1}\|_\sE
  }{\|\Delta_{0, \alpha}^{-1}\|_\sE} = A \leq c_\alpha\, \|D - D_0\|. $$
  The latter inequality finishes the proof of the theorem in the
  case~$\|D - D_0\|\leq 1$.

  The case~$\|D - D_0\| \geq 1$ is reduced to the setting above by
  considering the triple $$ \frac {D}{\|D - D_0\|},\ \ \frac {D_0}{\|D -
    D_0\|},\ \ \frac \alpha{\|D - D_0\|}. $$ Thus, the theorem is
  proved.
\end{proof}

Theorem~\ref{InverseResultWeakened} considerably
improves~\cite[Appendix~B, Proposition~10]{CarPhi1998-MR1638603}, where,
for the special case~$\sE = \sL^p$, $1< p < \infty$, the authors prove
the H\"older estimate, i.e.\ that $$ \| \Delta^{-1} -
\Delta_{0}^{-1}\|_p \leq c \, \|\Delta_{0}^{-1}\|_p \, \|D -
D_0\|^{\frac 12}, $$ provided~$\|D - D_0\|\leq 1$.

Finally, using Theorem~\ref{InverseResultWeakened}, we can improve
Theorem~\ref{QcommResultWeak}.

\begin{thm}
  \label{QcommResultWeakened}
  Let~$\alpha > 0$ and~$\Delta_{0, \alpha}^{-1} \in \sE$.  If\/~$\|D -
  D_0\|\leq 1$, then~$D \Delta_\alpha^{-1} - D_0 \Delta^{-1}_{0,\alpha}
  \in \sE$ and there is a constant~$c > 0$ independent of~$\alpha$ such
  that $$ \|D \Delta^{-1}_{\alpha} - D_0 \Delta^{-1}_{0, \alpha} \|_\sE
  \leq \, c\, \max \left\{1, \alpha^{-\frac 12} \right\}\,
  \|\Delta^{-1}_{0, \alpha}\|_\sE\, \|D - D_0\|.  $$
\end{thm}

\begin{proof}
  It follows from Theorem~\ref{QcommResultWeak} that
  \begin{equation}
    \label{QcommResultWeakenedTemp}
    \|D \Delta_\alpha^{-1} - D_0    \Delta_{0, \alpha}^{-1}\|_\sE \leq
    c'\, \|\Delta_{0, \alpha}^{-1}\|^{\frac 12}_\sE \,
    \|\Delta_{\alpha}^{-1}\|_\sE^{\frac 12} \, \|D - D_0\|,\ \ c' > 0.
  \end{equation}
  On the other hand, it follows from the assumption~$\|D - D_0\| \leq
  1$, Theorem~\ref{InverseResultWeakened} and triangle inequality that
  $$ \|\Delta_\alpha^{-1}\|_\sE \leq \|\Delta_{0, \alpha}^{-1}\|_\sE \left( 1 +
    c'' \, \max\left\{1, \alpha^{-1}\right\} \right),\ \ c''> 0. $$
  Replacing~$\|\Delta_\alpha^{-1}\|_\sE$ on the right
  in~\eqref{QcommResultWeakenedTemp} with the right-hand side of the
  latter inequality, we arrive at $$ \|D \Delta_{\alpha}^{-1} - D_0
  \Delta_{0, \alpha}^{-1}\|_\sE \leq \, c \, \max\left\{1,
    \alpha^{-\frac 12}\right\}\, \|\Delta_{0, \alpha}^{-1}\|_\sE \|D -
  D_0\|, $$ for some~$c > 0$.  The claim of the theorem is proved.
\end{proof}


\subsubsection{An application to the weak $L^p$ spaces}
\label{sec:AppWeakLp}

In view of its relevance to the definition of the Dixmier trace
the weak $L^p$-space~$\sL^{p, \infty}$ has come to play an important
role in noncommutative geometry. For this reason we describe in this Section
some consequences of our methods for this space. We note that by
specialising we obtain sharper estimates (it is possible
with additional effort to  establish weaker
analogous results for more general ideals $\sE$).

We shall improve Theorem~\ref{CommResult} in the special
setting of $\sL^{p, \infty}$.  For the sake of
brevity, we shall denote the norm in the latter space
by~$\|\cdot\|_{p, \infty}$, $1\leq p < \infty$.  Recall, that the
latter norm is given by $$ \|x\|_{p, \infty} = \sup_{t \geq 0}
t^{\frac 1p} \mu_t(x),\ \ x \in \sL^{p, \infty}. $$ We refer for
further detailed discussion of properties of the weak $L^p$-spaces
to~\cite{CaPhSu2000-MR1758245, CarPhi1998-MR1638603,
  Sukochev2000-MR1767406}.

\begin{thm}
  \label{AlanObservation}
  Let~$1\leq p < \infty$, $r > 1$, $\Delta^{-1} \in \sL^{p, \infty}$
  and~$a \in \aM$.  If\/~$[D, a] \in \sL^\infty$, then~$[D \Delta^{-r},
  a] \in \sL^p$, and there is a constant independent of~$r$ such that $$
  \|[D \Delta^{-r}, a]\|_p \leq \, c(p, r)\, \|[D, a]\|, $$ where $$
  c(p, r) \leq c\, \max\left\{1, [p\,(r-1)]^{-1/p}\,
    \|\Delta^{-1}\|^{\frac 12 r + \frac 12}_{p, \infty}\right\}. $$
\end{thm}

\begin{proof}
  We shall modify the argument given in the proof of
  Theorem~\ref{CommResult}.

  Fix~$r > 1$.  We also fix~$\epsilon = \frac 12 (r - 1)>0$.  Let us
  first note that, since the operator~$\Delta^{-1}$ is bounded, it
  readily follows from the definition of~$\sL^{p, \infty}$ that
  \begin{equation}
    \label{GinLi}
    \Delta^{-r + \epsilon} \in \sL^p \ \ \text{and}\ \  \|\Delta^{-r +
      \epsilon}\|_p \leq c\, \max\left\{1, [p\,(r-1)]^{-1/p}
      \|\Delta^{-1}\|^{r - \epsilon}_{p, \infty} \right\},
  \end{equation}
  for some numerical constant~$c > 0$.  Note the following simple
  identities $$ [D \Delta^{-r}, a] = D \Delta^{-\epsilon} [\Delta^{-r +
    \epsilon}, a] + [D \Delta^{-\epsilon} , a] \Delta^{-r + \epsilon} $$
  and $$ [\Delta^{1-r}, a] = \Delta^{1 - \epsilon} [ \Delta^{-r +
    \epsilon}, a] + [\Delta^{1 - \epsilon}, a] \Delta^{-r + \epsilon}.
  $$ Combining these two together we arrive at the following equation $$
  [D \Delta^{-r}, a] = [D \Delta^{-\epsilon}, a] \Delta^{-r + \epsilon}
  + D\Delta^{-1} ([\Delta^{1-r}, a] - [\Delta^{1 - \epsilon}, a]
  \Delta^{-r + \epsilon}). $$ Consequently, the claim of the theorem
  will follow if we show that
  \begin{equation}
    \label{AlanObservationClaimI}
    [D \Delta^{- \epsilon}, a],\
    [\Delta^{1 -\epsilon}, a] \in \sL^\infty \ \ \text{and}\ \
    [\Delta^{1 - r}, a] \in \sL^p.
  \end{equation}
  The first claim in~\eqref{AlanObservationClaimI} follows from
  Theorem~\ref{DOIcommutator} which is applicable here since the
  functions $$ f_1(t) = \frac t{(1 + t^2)^{\frac \epsilon 2}} \ \
  \text{and}\ \ f_2(t) = (1 + t^2)^{\frac {1 - \epsilon}2} $$ satisfy
  the assumptions of Theorem~\ref{HolderCriterion}.  To prove that
  \begin{equation}
    \label{EquivProposal}
    [\Delta^{1-r}, a] \in \sL^p,\ \ r > 1.
  \end{equation}
  We consider an infinitely smooth function~$\chi_0(t)$ which is~$1$
  when~$t \in [-1, 1]$; and~$0$ when~$t \not \in [-2, 2]$.  We
  set~$\chi_1 = 1 - \chi_0$.  Let~$f_r(t) = t^{1 - r}$, $t \in \Rl$ and
  take the representation
  \begin{align}
    \label{PsiResolution}
    \psi_{f_r}(\lambda, \mu) = &\, \chi_0\left ( \log \frac \lambda \mu
    \right) \, \frac{\mu^{1-r} \left ( 1 - \left ( \frac \lambda \mu
        \right )^{1-r} \right )} {\mu \left ( 1 - \left ( \frac \lambda
          \mu \right ) \right )} \cr + &\, \chi_1\left ( \log \frac
      \lambda \mu \right ) \, \frac {\lambda^{1-r} -
      \mu^{1-r}}{\lambda^{\frac 12} \mu^{\frac 12} \left ( \left (\frac
          \lambda \mu \right)^{\frac 12} - \left ( \frac \lambda \mu
        \right)^{- \frac 12} \right) } \cr = &\, \mu^{-r} \chi_0 \left
      (\log \frac \lambda \mu \right)\, \frac {1 - \left ( \frac \lambda
        \mu \right)^{1-r}}{1 - \frac \lambda \mu} \cr + &\, \chi_1\left
      (\log \frac \lambda \mu \right)\, \frac{\lambda^{\frac 12 - r}
      \mu^{- \frac 12} }{ \left ( \frac \lambda \mu \right )^{\frac 12}
      - \left ( \frac \lambda \mu \right )^{- \frac 12}} \cr - &\,
    \chi_1\left (\log \frac \lambda \mu \right)\, \frac{ \lambda^{-
        \frac 12} \mu^{\frac 12 - r} }{ \left ( \frac \lambda \mu \right
      )^{\frac 12} - \left ( \frac \lambda \mu \right )^{- \frac 12}}.
  \end{align}
  Noting that the functions $$ \chi_0(t)\, \frac {1 - e^{(1-r) t}}{1 -
    e^t} \ \ \text{and}\ \ \frac{\chi_1(t)}{e^{\frac t2} - e^{- \frac
      t2}} $$ and their first derivatives belong to~$L^2(\Rl)$, we now
  infer~\eqref{EquivProposal} from Lemmas~\ref{DOIbddCrCor},
  \ref{FourierLi} and~\ref{MaximumPrincCorollary} and
  Theorem~\ref{DOIcommutator}.
\end{proof}

\subsubsection{Applications to Fredholm modules and spectral flow}
\label{FredholmSec}

Let~$(\aM, \tau)$ be a semi-finite von Neumann algebra acting on a
separable Hilbert space~$\sH$ with a n.s.f.\ trace~$\tau$ and let~$\sE =
E(\aM, \tau)$ be noncommutative symmetric space.  Let~$\aA$ be a unital
Banach $*$-algebra which is represented in~$\aM$ via a continuous
faithful $*$-homomorphism~$\pi$.  We shall identify the algebra~$\aA$
with its representation~$\pi(\aA)$.  In `semifinite noncommutative
geometry' one studies the following objects (see~\cite{ScWaWa1998,
  Connes1985-MR823176, Connes1994-MR1303779, CarPhi1998-MR1638603,
  Sukochev2000-MR1767406, CaPhSu2000-MR1758245,CarPhi2004-MR2053481}).

\begin{deff}
  \label{SpectralTripleDef}
  {\it An odd $\sE$-summable semifinite spectral triple for~$\aA$}, is
  given by a triple~$(\aM, D_0,\aA)$, where~$D_0$ is an unbounded
  self-adjoint operator affiliated with~$\aM$ such that
  \begin{enumerate}
  \item $(1 + D_0^2)^{-\frac 12}$ belongs to~$\sE$;
  \item the subspace~$\aA_0$ given by $$ \aA_0 := \{ a \in \aA:\ \ [D_0,
    a] \in \aM \} $$ is a dense $*$-subalgebra of~$\aA$.
  \end{enumerate}
\end{deff}

\begin{deff}
  \label{FredholmDef}
  {\it An odd $\sE$-summable bounded pre-Breuer-Fredholm module
    for~$\aA$}, is given by a triple~$(\aM, F_0, \aA)$, where~$F_0$ is a
  bounded self-adjoint operator in~$\aM$ such that
  \begin{enumerate}
  \item $|1 - F^2_0|^{\frac 12}$ belongs to~$\sE$;
  \item the subspace~$\aA_\sE$ given by $$ \aA_\sE := \{ a
    \in \aA:\ \ [F_0, a] \in \sE \} $$ is a dense $*$-subalgebra
    of~$\aA$.
  \end{enumerate}
  If~$1 - F_0^2 = 0$,  the prefix ``pre-'' is dropped.
\end{deff}

\begin{corl}
  \label{CommResultCorl}
  If~$(\aM, D_0,\aA)$ is an odd semifinite $\sE$-summable spectral
  triple then~$(\aM, F_0,\aA)$, where~$F_0 = D_0 (1 + D^2_0)^{- \frac
    12}$ is an odd bounded $\sE$-summable pre-Breuer-Fredholm module.
  Furthermore, there is a constant~$c = c(F_0)$ such that
  \begin{equation}
    \label{CommResultCorlDisp}
    \|F - F_0\|_{\sE} \leq c\, \|D - D_0\|,
  \end{equation}
  for all~$D - D_0 \in \aM$, $\|D - D_0\| \leq 1$.
\end{corl}

\begin{proof} A straightforward application of Theorem~\ref{CommResult} shows
that for an arbitrary~$\aA$, we have $\aA_0\subseteq \aA_\sE$. An
application of Theorem~\ref{QcommResultWeakened}
implies~\eqref{CommResultCorlDisp}.

\end{proof}

The first assertion of Corollary~\ref{CommResultCorl} was also
proved in~\cite{CaPhSu2000-MR1758245} (see also~\cite{Sukochev2000-MR1767406})
in the special case when~$\sE =
\sL^p$, $1 < p < \infty$
(see~\cite[Theorem~0.3.(i)]{CaPhSu2000-MR1758245}) or when~$\sE$
is an interpolation space for a couple~$(\sL^p, \sL^q)$, $1< p
\leq q < \infty$ (see~\cite[Corollary~0.5]{CaPhSu2000-MR1758245}).
However, the methods employed in~\cite{CaPhSu2000-MR1758245},
\cite{Sukochev2000-MR1767406} and~\cite[Sections~A
and~B]{CarPhi1998-MR1638603} do not extend to an arbitrary
operator space~$\sE$ and more importantly, they do not yield the
Lipschitz estimate~\eqref{CommResultCorlDisp}, but only H\"older
estimates.  We now show that our methods not only yield the
Lipschitz continuity of the mapping~$(\aM, D_0) \mapsto (\aM,
F_0)$ but also its differentiability.

\begin{thm}
  \label{DiffCorl}
  Let~$\{D_t\}_{t \in \Rl}$ be a collection of self-adjoint linear
  operators affiliated with~$\aM$.  If\/~$(1 + D_0^2)^{-\frac 12}$ belongs to~$\sE$
   and if\/~$t \mapsto
  D_t$ is an operator-norm differentiable path at the point~$t = 0$,
  i.e.\ $D_t - D_0 \in \aM$, $t \in \Rl$ and there is an operator~$G
  \in \aM$ such that
  \begin{equation}
    \label{DiffCorlHyp}
    \lim_{t \rightarrow 0} \left \| \frac {D_t - D_0}t - G \right\| = 0,
  \end{equation}
  then the path~$t \rightarrow F_t := D_t (1 + D^2_t)^{- \frac 12}$ is
  $\sE$-differentiable at the point~$t = 0$, i.e.\ $F_t - F_0 \in \sE$
  and there is an operator~$H \in \sE$ such that $$ \lim_{t
    \rightarrow 0} \left \| \frac {F_t - F_0}t - H \right\|_\sE = 0.
  $$ Moreover, $H= T_{\psi_f}(G)$, where~$T_{\psi_f} = T_{\psi_f}
  (D_0, D_0)$ and
  \begin{equation}
    \label{DiffCorlFunc}
    f(t) = \frac t {(1 + t^2)^{\frac 12}}.
  \end{equation}
\end{thm}

\begin{proof}
  We shall use the argument from the proof of Theorems~\ref{CommResult},
  \ref{QcommResultWeak} and~\ref{QcommResultWeakened} in the special
  case when~$\alpha = 1$ and~$\theta = \frac 12$.  We let~$c$ stand for
  a positive constant which may vary from line to line.  We
  set~$\Delta_t := (1 + D^2_t)^{ \frac 12}$.  We start again with the
  identity (see Theorem~\ref{DOIcommutator})
  \begin{equation}
    \label{DiffCorlTemp}
    F_t - F_0 = T_{\psi_f}(D_t - D_0),\ \ t \in \Rl,
  \end{equation}
  where~$T_{\psi_f} = T_{\psi_f}(D_t, D_0)$.  An inspection of the proof
  of Theorem~\ref{CommResult} (see formulae~\eqref{TalphaDecomp},
  \eqref{CommResultLastTmp} and~\eqref{psiDecomp}) shows that
  \begin{equation}
    \label{DiffCorlForCiOld}
    T_{\psi_f}(x) =
    \Delta_t^{- \frac 12} \, T_t(x)\, \Delta_0^{-\frac 12},\ \ x \in
    \sL^\infty,
  \end{equation}
  where
  \begin{equation}
    \label{DiffCorlTmp}
    T_t(x) = T'_t(x) +
    \Delta_t^{-1} \, T'_t(x)\, \Delta_0^{-1} - D_t \Delta_t^{-1}\,
    T'_t(x)\, D_0 \Delta_0^{-1},
  \end{equation}
  and where~$T'_t = T_{\phi'}(D_t, D_0)$, $t \in \Rl$,
  \begin{equation}
    \label{AnotDecomSpecial}
    \phi'(\lambda, \mu) = \frac
    {(1+\lambda^2)^{\frac 14} (1+\mu^2)^{\frac 14}}{(1+
      \lambda^2)^{\frac 12} + (1+ \mu^2)^{\frac 12}} = \int_{\Rl} h(s)\,
    (1+\lambda^2)^{\frac{is}2} \, (1 + \mu^2)^{-\frac {is}2}\, ds,
  \end{equation}
  for some~$h \in L^1(\Rl)$.  The last equality
  in~\eqref{AnotDecomSpecial} follows from Lemma~\ref{DOIbddCrCor}
  (a further inspection shows that~$h = \hat g$, where~$g$ is
  given~\eqref{TheFirstGfunc}).  It is clear
  from~\eqref{AnotDecomSpecial} and Corollary~\ref{AzeroClass} that the
  operator~$T'_t \in \Bd(\sL^\infty)$ uniformly for~$t \in \Rl$.  Thus,
  by~\eqref{DiffCorlTmp} the operator~$T_t \in \Bd(\sL^\infty)$
  uniformly for~$t \in \Rl$.

  Our first objective is to prove that
  \begin{equation}
    \label{DiffCorlTempI}
    \lim_{t \rightarrow 0} \|T_t(x) - T_0(x)\| =
    0,\ \ x \in \sL^\infty.
  \end{equation}
  However, instead of proving~\eqref{DiffCorlTempI} we shall show a
  stronger result.  Let~$\{D'_t\}_{t \in \Rl}$ be another collection
  of linear self-adjoint operators affiliated with~$\aM$ and let~$\bar
  T_{\psi_f} = T_{\psi_f}(D_t, D'_s)$, $t, s \in \Rl$.  If~$\bar
  T'_{t, s} = T_{\phi'}(D_t, D'_s)$, $t, s \in \Rl$, then, alongside
  with~\eqref{DiffCorlForCiOld} and~\eqref{DiffCorlTmp}, we have
  \begin{equation}
    \label{DiffCorlForCi}
    \bar T_{\psi_f} (x) = \Delta_t^{- \frac 12}\, \bar T_{t,s}(x)\,
    \Delta_s'^{-\frac 12},
  \end{equation}
  where
  \begin{equation}
    \label{DiffCorlTmpStronger}
    \bar T_{t, s}(x) = \bar T_{t, s}'(x) + \Delta_t^{-1}\, \bar T_{t,s}'(x)\,
    \Delta_s'^{-1} - D_t \Delta^{-1}_t \, \bar T_{t,s}'(x)\, D_s'
    \Delta_s'^{-1}
  \end{equation}
  and~$\Delta_t' = (1 + (D_t')^2)^{\frac 12}$.  We shall show
  \begin{equation}
    \label{DiffCorlTempIStronger}
    \lim_{t \rightarrow 0} \|\bar T_{t,t}(x) - \bar T_{0,t}(x)\| = 0,\ \ x \in
    \sL^\infty
  \end{equation}
  which in particular implies~\eqref{DiffCorlTempI}.  It obviously
  follows from~\eqref{DiffCorlHyp} that $$ \lim_{t \rightarrow 0}
  \|D_t - D_0\| = 0. $$ Consequently, it readily follows from
  Theorems~\ref{InverseResultWeakened} and~\ref{QcommResultWeakened}
  that
  \begin{equation}
    \label{DiffCorlForCiNext}
    \lim_{t \rightarrow 0} \|\Delta_t^{-1} - \Delta_0^{-1}\| =
    \lim_{t \rightarrow 0} \|D_t \Delta_t^{-1} - D_0 \Delta_0^{-1}\| = 0.
  \end{equation}
  Combining the latter with~\eqref{DiffCorlTmpStronger}, it is seen
  that to show~\eqref{DiffCorlTempIStronger}, we need only to prove that
  \begin{equation}
    \label{DiffCorlTmpI}
    \lim_{t \rightarrow 0} \|\bar T'_{t,t}(x) - \bar T'_{0,t}(x)\| = 0,\ \ x\in
    \sL^\infty.
  \end{equation}
  For the latter, fix~$\epsilon > 0$.  We also fix~$\delta > 0$ such
  that~$\|D_t - D_0\| < \epsilon$ for every~$|t| < \delta$.  It is
  clear that there is~$s_0 > 0$ and the function
  \begin{equation}
    \label{phiTwoPrimes}
    \phi''(\lambda,
    \mu) = \int_{|s| \leq s_0} h(s)\, (1 + \lambda^2)^{\frac {is}2}\, (1
    + \mu^2)^{-\frac {is}2}\, ds
  \end{equation}
  such that
  \begin{equation}
    \label{DiffCorlTmpII}
    \|\phi' - \phi''\|_{\mathfrak{A}_0} < \epsilon.
  \end{equation}
  Furthermore, for every fixed~$|s| \leq s_0$, the function $$ f_s(t)
  = (1 + t^2)^{\frac {is}2} $$ satisfies Theorem~\ref{HolderCriterion}
  with constants depending only on~$s_0$ and therefore, by
  Corollary~\ref{AzeroClass} and Theorem~\ref{DOIcommutator},
  \begin{equation}
    \label{DiffCorlTmpIII}
    \| \Delta_t^{is} - \Delta_0^{is} \| \leq c\, \|D_t - D_0\|,\ \ |s|
    \leq s_0.
  \end{equation}
  Let us show the identity
  \begin{align}
    \label{SeemsTheLastTmp}
    \bar T'_{t, t}(x) - \bar T'_{0, t}(x) = &\, (\bar T'_{t, t}(x) -
    T''_{t, t}(x)) + (T''_{0, t}(x) - \bar T'_{0, t}(x)) \cr + &\,
    \int_{|s|\leq s_0} h(s)\, (\Delta_t^{is} - \Delta_0^{is})\, x \,
    (\Delta_t')^{is}\, ds,
  \end{align}
  where~$T''_{t,s} = T_{\phi''}(D_t, D_s')$, $t, s \in \Rl$.  Fix~$x
  \in \sL^2 \cap \sE$ and~$y \in \sL^2 \cap \sEcross$.  We set $$
  d\nu_{t,s} = \tau(y\, dE_{t, \lambda}\, x\, dE'_{s, \mu}),\ \ t,s \in \Rl,
  $$ where~$dE_{t, \lambda}$ and~$dE'_{s, \mu}$ are the spectral
  measure of the operator~$D_t$ and~$D'_s$, respectively.
  By~\eqref{DOIdefFmla} we have
  \begin{align}
    \label{SeemsTheLastTmpI}
    \tau(y \bar T'_{t,t}(x)) - \tau(y \bar T'_{0,t}(x)) = &\, \int_{\Rl^2}
    \phi'\, d\nu_{t,t} - \int_{\Rl^2} \phi'\, d\nu_{0,t} \cr = &\, \int_{\Rl^2}
    (\phi' - \phi'') \, d \nu_{t,t} + \int_{\Rl^2} (\phi' - \phi'')\, d\nu_{0,t}
    \cr + &\, \int_{\Rl^2} \phi''\, d\nu_{t,t} - \int_{\Rl^2} \phi''\,
    d\nu_{0,t}.
  \end{align}
  Replacing~$\phi''$ with~\eqref{phiTwoPrimes} and Fubini's theorem
  yield for the last term
  \begin{multline}
    \label{SeemsTheLastTmpII}
    \int_{\Rl^2} \phi''\, d\nu_{t, t} - \int_{\Rl^2} \phi''\, d\nu_{0, t} \cr = \,
    \int_{|s|\leq s_0} h(s)\, ds\, \left[ \int_{\Rl^2} (1+
      \lambda^2)^{\frac {is}2} \, (1 + \mu^2)^{- \frac {is}2}\, d(\nu_{t, t}
      - \nu_{0, t}) \right] \cr = \, \int_{|s|\leq s_0} h(s)\, ds\, \left[
      \tau(y\, \Delta_t^{is} \, x \, (\Delta_t')^{-is}) - \tau(y\,
      \Delta_0^{is}\, x \, (\Delta_t')^{-is}) \right].
  \end{multline}
  Here, we used the spectral theorem as follows
  \begin{multline*}
    \int_{\Rl^2} (1 + \lambda^2)^{\frac{is}2} \, (1 + \mu^2)^{-
      \frac{is}2} \, d\nu_{t,s} \cr = \, \int_{\Rl} \int_{\Rl} (1 +
    \lambda^2)^{\frac{is}2}\, (1 + \mu^2)^{- \frac{is}2}\, \tau(y\,
    dE_{t, \lambda}\, x \, dE'_{s, \mu}) \cr = \, \tau \left[ y\,
      \int_{\Rl} (1 + \lambda^2)^{\frac{is}2}\, dE_{t, \lambda}\, x \,
      \int_{\Rl} (1 + \mu^2)^{-\frac{is}2}\, dE'_{s, \mu} \right] =
    \, \tau(y\, \Delta_t^{is}\, x \, (\Delta_s')^{-is}).
  \end{multline*}
  Hence, combining~\eqref{SeemsTheLastTmpII}
  with~\eqref{SeemsTheLastTmpI} yields~\eqref{SeemsTheLastTmp}.

  Estimating the first two terms in~\eqref{SeemsTheLastTmp}
  with~\eqref{DiffCorlTmpII} and the last one
  with~\eqref{DiffCorlTmpIII} yields $$
  \|\bar T'_{t, t}(x) - \bar T'_{0, t}(x)\|
  \leq \epsilon\, (2 + c\, \|h\|_{L^1})\, \|x\|, $$
  provided~$|t|<\delta$.  The latter finishes the proof
  of~\eqref{DiffCorlTmpI}.

  Our next objective is to prove that
  \begin{equation}
    \label{DiffCorlTempIII}
    \lim_{t \rightarrow 0} \left\| \Delta_t^{- \frac 12} -
      \Delta_0^{-\frac 12}\right\|_{\sE^{(2)}} = 0,
  \end{equation}
  where~$\sE^{(2)}$ is the $2$-convexification of~$\sE$ (see
  e.g.~\cite{LT-II}), i.e. $$ \sE^{(2)} := \{x \in \aMtilde:\ \ |x|^2
  \in \sE\} \ \ \text{and}\ \ \|x\|_{\sE^{(2)}} := \|\,
  |x|^2\|_{\sE}^{\frac 12}. $$ We have
  \begin{align*}
    \left\|\Delta_t^{- \frac 12} - \Delta^{- \frac
        12}_0\right\|_{\sE^{(2)}} \leq &\, \left \| \Delta_t^{-\frac 12}
    \right\|\, \left\| \Delta_0^{- \frac 12}\right\|_{\sE^{(2)}}\, \left
      \|\Delta_t^{\frac 12} - \Delta_0^{\frac 12} \right\| \cr \leq &\,
    \left\| \Delta_0^{-1}\right\|_{\sE}^{\frac 12}\, \left
      \|\Delta_t^{\frac 12} - \Delta_0^{\frac 12} \right\| \cr \leq &\,
    c \, \|\Delta_0^{-1}\|_\sE^{\frac 12}\, \| D_t - D_0 \|,\ \ c >
    0.
  \end{align*}
  Here, the first inequality follows from the simple observation that
  $$ \Delta_t^{- \frac 12} - \Delta_0^{-\frac 12} = - \Delta_t^{-
    \frac 12}\,(\Delta_t^{\frac 12} - \Delta_0^{\frac 12})\, \Delta_0^{-
    \frac 12}. $$ The last inequality follows from
  Theorem~\ref{DOIcommutator} and the fact that the function $$ h(t) =
  (1 + t^2)^{\frac 14} $$ satisfies Theorem~\ref{HolderCriterion}.

  Now we can finish the proof of the theorem.  We set~$H =
  T_{\psi_f}(G) = \Delta_0^{-\frac 12} T_0(G)\, \Delta_0^{- \frac
    12}$.  It follows from~\eqref{DiffCorlTemp}, that
  \begin{align*}
    \frac {F_t - F_0}t - H = &\, \Delta_t^{-\frac 12}\, T_t\left( \frac
      {D_t - D_0}t \right) \, \Delta_0^{- \frac 12} - \Delta_0^{-\frac
      12} \, T_0(G)\, \Delta_0^{-\frac 12} \cr = &\, (\Delta_t^{- \frac
      12} - \Delta_0^{-\frac 12})\, T_t\left ( \frac {D_t - D_0}t \right)\,
      \Delta_0^{-\frac 12} \cr + &\, \Delta_0^{-\frac 12}\, T_t \left(
        \frac {D_t - D_0}t - G \right) \, \Delta_0^{-\frac 12} \cr +
      &\, \Delta_0^{-\frac 12}\, \left( T_t(G) - T_0(G) \right)\,
      \Delta_0^{-\frac 12}.
  \end{align*}
  Note that the operators~$T_t$ and~$\frac {D_t - D_0}t$ are uniformly
  bounded for every~$t \in \Rl$ in the spaces~$\Bd(\sL^\infty)$
  and~$\sL^\infty$, respectively (see remarks
  following~\eqref{AnotDecomSpecial}).  Therefore, when~$t \rightarrow
  0$, the first term vanishes in~$\sE$ due to~\eqref{DiffCorlTempIII}
  and generalized H\"older inequality~$\|xy\|_{\sE} \leq
  \|x\|_{\sE^{(2)}}\, \|y\|_{\sE^{(2)}}$, $x, y \in \sE^{(2)}$; the
  second term vanishes in~$\sE$ due to~\eqref{DiffCorlHyp} and
  Lemma~\ref{MaximumPrincCorollary}; and the last one does the same
  thanks to~\eqref{DiffCorlTempI} and Lemma~\ref{MaximumPrincCorollary}.
  The theorem is proved.
\end{proof}

The result above is of importance for the spectral flow theory, for
which we refer to~\cite{CarPhi1998-MR1638603, CarPhi2004-MR2053481,
  BCPRSW}. In that theory, given an odd $\sE$-summable spectral triple
(respectively, bounded $\sE$-summable pre-Breuer-Fredholm
module,~$(\aM, F_0,\aA)$) one introduces an associated affine
space $\Phi_\sE:=\{D=D_0+A\ |\ A=A^*\in \aM\}$ (respectively,
$\mathfrak{M}_\sE:=\{F=F_0+A\ |\ A=A^*\in \sE\}$).  To compare the
spectral flow along paths of self-adjoint Breuer-Fredholm
operators in~$\mathfrak{M}_\sE$ and self-adjoint bounded operators
in~$\Phi_\sE$ it is important to know that the transformation from
spectral triples $(\aM, D_0,\aA)$ to bounded $\sE$-summable
modules~$(\aM, F_0,\aA)$ via the map~$F_D=D(1+D^2)^{-1/2}$ carries
$C^1$~paths to $C^1$~paths (see
e.g.~\cite[Section~6]{CarPhi2004-MR2053481}).  In concrete
examples, proving the smoothness of this map is difficult and such
a difficulty has led to extra technical assumptions imposed
in~\cite{CarPhi1998-MR1638603, CarPhi2004-MR2053481} on the
triple~$(\aM, D_0,\aA)$.  The result below removes all such
assumptions.

\begin{thm}
  \label{CiPATH}
  Let~$\{D_t\}_{t \in \Rl}$ be a collection of self-adjoint linear
  operators affiliated with~$\aM$ such that~$(1 + D_0^2)^{-\frac 12}$ belongs to~$\sE$.
  Let~$F_t := D_t\, (1 +
  D_t^2)^{-\frac 12}$, $t \in \Rl$.  If~$D_t - D_{t_0} \in \sL^\infty$,
  $t, t_0 \in\Rl$, the limit
  $$ \frac {d D_t}{dt}(t_0) := \|\cdot\| - \lim_{t \rightarrow t_0}
  \frac {D_t - D_{t_0}}{t - t_0},\ \ t_0 \in \Rl $$ exists and the
  mapping $t \mapsto \frac{d D_t}{dt}(t)$ is operator norm continuous,
  then~$F_t - F_{t_0} \in \sE$, the limit $$ \frac {d F_t}{dt}(t_0) :=
  \|\cdot\|_{\sE} - \lim_{t \rightarrow t_0} \frac {F_t - F_{t_0}}{t -
    t_0} $$ exists and the mapping~$t \rightarrow \frac{d F_t}{dt}(t)$
  is $\sE$-continuous.
\end{thm}

\begin{proof}
  The existence of the derivative~$\frac{d F_t}{dt}$ follows from
  Theorem~\ref{DiffCorl}.  We need only to show that the
  derivative~$\frac {dF_t}{dt}$ is $\sE$-continuous.  Clearly, it is
  sufficient to show the continuity at~$t = 0$, i.e.\ we need to show
  that $$
  \lim_{s \rightarrow 0} \left \| \frac{dF_t}{dt}(s) -
    \frac{dF_t}{dt}(0) \right\|_\sE = 0. $$
  It follows from
  Theorem~\ref{DiffCorl} that $$
  \frac{dF_t}{dt} (s) = T_{\psi_f}(D_s,
  D_s)\, \frac{dD_t}{dt} (s),\ \ s \in \Rl, $$
  where~$f$
  from~\eqref{DiffCorlFunc}.  Consequently, we have
  \begin{align}
    \label{CiPATHtmp}
    \frac{dF_t}{dt}(s) - \frac{dF_t}{dt}(0) = &\, T_{\psi_f}(D_s, D_s)\,
    \frac{dD_t}{dt}(s) - T_{\psi_f}(D_0, D_0)\, \frac{dD_t}{dt}(0) \cr =
    &\, T_{\psi_f}(D_s, D_s) \left[ \frac{dD_t}{dt}(s) -
    \frac{dD_t}{dt}(0) \right] \cr + &\, T_{\psi_f}(D_s, D_s)\,
    \frac{dD_t}{dt}(0) - T_{\psi_f}(D_s, D_0)\, \frac{dD_t}{dt} (0) \cr
    + &\, T_{\psi_f}(D_s, D_0)\, \frac{dD_t}{dt}(0) - T_{\psi_f}(D_0,
    D_0)\, \frac{dD_t}{dt} (0).
  \end{align}
  By~\eqref{DiffCorlForCiNext} and Theorem~\ref{CommResult} the
  operator~$T_{\psi_f}(D_s, D_s)$ is bounded in~$\Bd(\sL^\infty, \sE)$
  and there are constants~$c > 0$ and~$\delta > 0$, such that $$
  \|T_{\psi_f}(D_s, D_s)\|_{\Bd(\sL^\infty, \sE)} \leq c,\ \ |s| <
  \delta. $$ Thus, the first term in~\eqref{CiPATHtmp} vanishes
  in~$\sE$ as~$t \rightarrow 0$ since~$\frac{dD_t}{dt}$ is operator
  norm continuous.  On the other hand, the last two terms
  in~\eqref{CiPATHtmp} vanish in~$\sE$ when~$t \rightarrow 0$
  thanks to~\eqref{DiffCorlForCi}, \eqref{DiffCorlTempIStronger},
  \eqref{DiffCorlForCiNext}, and~\eqref{DiffCorlTempIII}.
\end{proof}

{\baselineskip=0.7\baselineskip \small
  \let\Large=\normalsize \bibliography{references}}

\end{document}